\documentclass[12pt]{article}
\usepackage{graphicx} 

\usepackage{tikz}

\usepackage{amsmath,amsthm,amscd,amssymb,eucal,mathrsfs}
\usepackage{amsfonts}
\usepackage{latexsym}
\usepackage{graphicx} 
\usepackage{mathtools}
\usepackage{multicol}
\usepackage{dsfont}
\usepackage[all]{xy}
\usepackage{multirow}
\usepackage{color}
\usepackage{mathtools}
\usepackage[hidelinks]{hyperref}
\usepackage[all]{xy}

\newtheorem{thm}{Theorem}[section]

\newtheorem{remark}[thm]{Remark}

\newtheorem*{Satz*}{Satz}

\newtheorem{Lemma}[thm]{Lemma}

\newtheorem{Corollary}[thm]{Corollary}

\newtheorem{Assumption}{Assumption}

\newcommand{\mathset}[1]{{\left\{#1\right\}}}
\newcommand{\absolute}[1]{\left\lvert#1\right\rvert}
\newcommand{\norm}[1]{\left\|#1\right\|}

\DeclareMathOperator{\PGL}{PGL}
\DeclareMathOperator{\SL}{SL}

\DeclareMathOperator{\dom}{dom}

\DeclareMathOperator{\range}{Ran}

\title{Schottky-Invariant $p$-Adic Diffusion Operators
}
\author{Patrick Erik Bradley}
\date{\today}

\begin{document}

\maketitle

\begin{abstract}
A parametrised diffusion operator on the regular domain $\Omega$ of a $p$-adic Schottky group is constructed. It is defined as an integral operator on the complex-valued functions on $\Omega$ which are invariant under the Schottky group $\Gamma$, where integration is against the measure defined by an invariant regular differential $1$-form $\omega$.  It is proven that the space of Schottky invariant $L^2$-functions on $\Omega$ outside the zeros of $\omega$ has an orthonormal basis consiting of $\Gamma$-invariant extensions of Kozyrev wavelets which are eigenfunctions of the operator. The eigenvalues are calculated, and it is shown that the heat equation for this operator provides a unique solution for its Cauchy problem with Schottky-invariant continuous initial conditions supportes outside the zero set of $\omega$, and gives rise to a strong Markov process on the corresponding orbit space for the Schottky group whose paths are c\`adl\`ag.
\end{abstract}

\section{Introduction}

The first diffusion operators on  $p$-adic domains are Vladimirov-Taibleson operators \cite{Taibleson1975,VVZ1994}. These are convolution operators on non-archimedean local fields, and hence diagonalisable by the Fourier transform. 
From this as a starting point, they were extended to the ad\`eles, and their connection to integration on path spaces via Feynman-Kac formulas was explored, including proofs that such types of diffusion are scaling limits, cf.\ e.g.\ \cite{Weisbart2021,Urban2024,Weisbart2024}.
As a $p$-adic ball is itself a compact abelian group, the Fourier transform method can be adapted to that case in the study of the heat equation \cite{Kochubei2018}.
\newline

Of importance is also their representation as a Laplacian integral operator. This allows the extension to compact $p$-adic subdomains which are not necessarily endowed with a group structure, and where Turing patterns can be observed
\cite{ZunigaNetworks,Zuniga2022}. Also, certain compact $p$-adic manifolds known as Mumford curves became amenable to their own diffusion operators in integral form \cite{brad_HeatMumf}.
\newline

The spectrum of such Laplacian operators can be studied via Kozyrev wavelets, introduced in order to find an orthonormal basis of the Hilbert space $L^2(\mathds{Q}_p)$ 
consisting of eigenfunctions of the Vladimirov operator \cite{Kozyrev2002,Kozyrev2004}.
These turned out to be extendable to Mumford curves, \cite{brad_HeatMumf}.
And in recent work, efforts were made in order to rid the constructions on Mumford curves from their dependence on a fundamental domain. Whereas in \cite{brad_HeatMumf}, the construction is exclusively on a compact fundamental domain, in the case of genus one, 
theta functions are used in order to construct an invariant meromorphic function
\cite{brad_thetaDiffusionTateCurve}, 
and this method was also extended to higher genus \cite{Brad_HearingGenusMumf}, 
allowing to hear the genus of a Mumford curve from the spectrum of a diffusion operator.
\newline

The main goal of this article is 
to not require the removal of an essential part of a given Mumford curve by resorting to a meromorphic function as previously done in \cite{Brad_HearingGenusMumf}, where it was necessary to exclude the pre-image of the limit set of the Schottky group under taking differences $x-y$ of two variables $x,y$ in the regular domain of the Schottky group action. And this is in general a set of positive measure. In the new approach here, only the zeros of a regular differential $1$-form need to be removed, and these form a zero set, thus obtaining a diffusion on a given Mumford curve almost everywhere through a Schottky invariant diffusion operator almost everywhere on the domain of regularity of the Schottky group.
This is obtained by simply taking the kernel function locally as a  positive power of the
$p$-adic absolute difference $\absolute{\beta x-\gamma y}^\alpha$, with $\beta,\gamma$ taken from the Schottky group, appropiately weighted by a function of the length  $\ell(\beta^{-1}\gamma)$ of the word $\beta^{-1}\gamma$ in a given set of $g$ generators of the Schottky group, and their inverses as a reference alphabet.
The only set now which needs to be excluded are the zeros of the invariant regular differential $1$-form $\omega$ giving rise to the measure $\absolute{\omega}$ on the regular domain $\Omega(K)$ of the Schottky group $\Gamma$. The main results of this article can be stated as follows:
\newline

\noindent
{\bf Theorem 1.}\emph{
The space of $\Gamma$-invariant $L^2$-functions on $\Omega$ outside the zeros of $\omega$ has an orthonormal basis consisting of $\Gamma$-invariant extensions of Kozyrev
wavelets supported on discs outside the zero set of $\omega$.
These are eigenfunctions of the self-adjoint diffusion operator $-\Delta_\alpha^\frac12$ on that Hilbert space. The eigenvalue corresponding to such a wavelet $\psi_{B,j}$, where $B$ is a disc, and $j$ an element of the residue field of the non-archimedean local field $K$, is
\[
\lambda_B=\mu^\Gamma(F)^{-1}
\sum\limits_{\gamma\in\Gamma}\absolute{\pi}^{\alpha_g\ell(\gamma)}
\left(
\int_{F\setminus B}
\absolute{x-\gamma y}^{-\alpha}\absolute{\omega(y)}+\mu^\Gamma(B)^{1-\alpha}
\right)
\]
with $\alpha>0$, and
depending on $B$ and a good fundamental domain $F$, and for any $\gamma\in\Gamma$, and where 
\[
\mu^\Gamma(C)=\int_C\absolute{\omega}
\] 
for any $\absolute{\omega}$-measurable set $C$.
The eigenvalues have finite multiplicities, and are invariant under shifting from a given fundamental domain $F$ to $\gamma F$ with $\gamma\in\Gamma$.
}
\newline

The technical notion of ``good fundamental domain'' is introduced in \cite[I.4.1.3]{GvP1980}, whose existence is guaranteed for any Mumford curve. 
Also, a straightforward transition formula for the eigenvalues under the replacement $F\to\gamma F$ with $\gamma\in
\Gamma$ is given in Lemma 4.12 below,  because there could be a possible effect of the actual arrangement the ``holes'' cut out of a disc in order to form a good fundamental domain.
\newline

The next theorem deals with the Cauchy problem for the heat equation
\[
\left(\frac{\partial}{\partial t}+\Delta_\alpha^{\frac12}\right)h(t,x)=0
\]
having initial condition $h(0,x)=h_0(x)$ wich is a continuous $\Gamma$-invariant function defined on $\Omega(K)\setminus V(\omega)$,  where $V(\omega)$ is the zero set of the differential form $\omega$, and the solution space is assumed to be
\[
C^1((0,\infty),\Omega(K)\setminus V(\omega))^\Gamma
\]
where the superscript $^\Gamma$ denotes that the functions are assumed invariant under the action of $\Gamma$.
\newline

\noindent
{\bf Theorem 2.}
\emph{
The heat equation for operator $-\Delta_\alpha^\frac12$ provides a unique solution for its Cauchy problem with $\Gamma$-invariant continuous initial condition $h_0(x)$ supported outside the zero set $V(\omega)$ of $\omega$, and is given as
\[
h(t,x)=
\int_{\Omega(K)\setminus V(\omega)}
h_0(y)p_t(x,\absolute{\omega(y)})
\]
given by a probability measure $p_t(x,\cdot)$ on the Borel $\sigma$-algebra on $\Omega(K)\setminus V(\omega)$, which is 
also the transition function of a strong Markov process on the orbit space $(\Omega(K)\setminus V(\omega))/\Gamma$ whose paths are c\`adl\`ag.
}

\section{Notation and Concepts Used}
Assume that $K$ is a $p$-adic number field, i.e.\ a finite extension of the field $\mathds{Q}_p$ of $p$-adic numbers. Denote the Haar measure on $K$ as $\mu_K$, or as $\absolute{dx}$ if the dependence on a variable $x$ is to be emphasised. It is normalised such that $\mu_K(O_K)=1$, where $O_K$ is the ring of integers of $K$. 
The absolute value on $K$ is denoted as $\absolute{\cdot}$, and is chosen such that
\[
\absolute{\pi}=p^{-f}
\]
where $\pi$ is a uniformiser of $O_K$, and $f$ is the degree of the residue field $O_K/\pi O_K$ as an extension of the finite field $\mathds{F}_p$ with $p$ elements. Indicator functions will often be written as
\[
\Omega(x\in B)=\begin{cases}
1,&x\in B\\
0,&x\notin B
\end{cases}
\]
where $B$ is a Borel measurable subset of $K$.
\newline

Any $n$-dimensional $K$-analytic manifold $X$ with a regular differential $n$-form $\omega$ has a measure
$\absolute{\omega}$ on $X$ outside the vanishing set $V(\omega)$ in $X$,
which locally on $U\subset X$ has the form
\[
\absolute{\omega}|_U=\absolute{f}\absolute{\mu_K}
\]
with $f\colon U\to K$  an analytic function, cf.\ \cite[Ch.\ II.2.2]{WeilAAG}, or \cite[Ch.\ 7.4]{IgusaLocalZeta}. 
Unlike in \cite{Brad_HearingGenusMumf}, the measure $\absolute{\omega}$ will not be extended to  $V(\omega)$, here. This exceptional set  is a zero set according to \cite[Lem.\ 3.1]{Yasuda2017}. More about $K$-analytic manifolds can be learned in \cite{Serre1992} or \cite{Schneider2011}, if the reader wishes so.
\newline

A Mumford curve can be viewed as a $1$-dimensional compact $K$-analytic manifold $X$ having an atlas consisting of pieces bi-analytically isomorphic to holed discs in $K$. They are explained in depth e.g.\ in \cite{GvP1980}. What is needed for this article is that they have a universal covering space which is open in the projective line $\mathds{P}^1(K)$, and the topological fundamental group $\Gamma$ of $X$ is a free group generated by $g$ hyperbolic M\"obius transformations in $\PGL_2(K)$, where $g$ is the genus of $X$. The group $\Gamma$ is also known as a so-called Schottky group.
A Mumford curve is also a projective algebraic curve defined over $K$, and posseses regular differential $1$-forms which are in fact algebraic. Namely,
according to \cite[Prop.\ VI.4.2]{GvP1980}, the space of regular differential $1$-forms on a Mumford curve of genus $g$ has dimension $g$ over the ground field $K$.
A regular algebraic differential $1$-form on the $K$-rational points $X(K)$ of $X$ is given by a $\Gamma$-invariant holomorphic  differential $1$-form
on $\Omega(K)$, where $\Omega\subset\mathds{P}^1_K$ is the universal covering space of $X$ which exists as an open analytic domain, cf.\  \cite[Ch.\ IV.3]{GvP1980}.

\begin{Assumption}\label{assumption}
It is assumed that the differential $1$-form $\omega\in\Omega_{X/K}^1$ has all its zeros in $X(K)$, the set of $K$-rational points of the Mumford curve $X$.
\end{Assumption}

This assumption can be fulfilled for a given algebraic differential $1$-form after a finite extension of the field $K$, if necessary.
The $\Gamma$-invariant differential $1$-form corresponding to $\omega$ of Assumption \ref{assumption} is also denoted as $\omega$. This should not be a cause for confusion, as the points of the Mumford curve $X$ themselves are $\Gamma$-orbits.
\newline

Let $L^2(\Omega(K),\absolute{\omega})$ be the Hilbert space of $L^2$-functions on $\Omega(K)$, on which the inner product 
\[
\langle f,g\rangle_\omega=\int_{\Omega(K)} f(x)\overline{g(x)}\,\absolute{\omega(x)}
\]
is used. The space of continuous functions on $\Omega(K)$ is denoted as $C(\Omega(K),\norm{\cdot}_\infty)$, and is a Banach space w.r.t.\ the supremum norm $\norm{\cdot}_\infty$.
\newline

Let $\mathscr{F}(\Omega(K))$ be a space of functions
$\Omega(K)\to\mathds{C}$
and define
\[
\mathscr{F}(\Omega(K))^\Gamma=\mathset{u\in\mathscr{F}\mid \forall\gamma\in \Gamma\;\forall x\in\Omega(K)
\colon u(\gamma x)=u(x)}
\]
as the corresponding subspace of $\Gamma$-invariant functions.
\newline

Similarly, a corresponding notation will be used for function spaces on $\Omega(K)\setminus V(\omega)$, where $V(\omega)\setminus\Omega(K)$ denotes the vanishing set of $\omega$. Since the differential $1$-form is algebraic as a differential form on $X$, this vanishing set is countable. An example is the space $L^2(\Omega(K)\setminus V(\omega))^\Gamma$.

\section{Kernel function for $\Gamma$-invariant functions}

Let $\Gamma=\langle\gamma_1,\dots,\gamma_g\rangle\subset\PGL_2(K)$ be a Schottky group on $g\ge 1$ generators with $K$ a non-archimedean local field. 
As an abstract group, $\Gamma$ is isomorphic to the free group $F_g$ with $g$ generators. Each element of $F_g$ can be uniquely represented as a reduced word over the alphabet $\mathset{\gamma_1^{\pm 1},\dots,\gamma_g^{\pm 1}}$, i.e.\ by deleting all expressions of the form
\[
\gamma_i\gamma_i^{-1}=1\quad\text{or}\quad
\gamma_i^{-1}\gamma_i=1
\]
for $i=1,\dots,g$. The length of a reduced word $w$ over a finite alphabet $\mathcal{A}$ is defined as the sum of the occurrence counts of each letter from $\mathcal{A}$ in $w$, and is denoted as $\ell(w)$.
\newline

The following result is well-known:

\begin{Lemma}\label{ReducedWordCount}
Fix $\beta\in\Gamma$. The number of elements $\gamma\in\Gamma$ 
such that $\beta^{-1}\gamma$ has  length $\ell$ is at most
\[
2g(2g-1)^{\ell(\beta)+\ell}
\]
for any natural number $\ell>0$.
\end{Lemma}

\begin{proof}
Assume first that $\beta=1$. Then
any of the $2g$ letters in $\gamma_1^{\pm 1},\dots,\gamma_g^{\pm1}$ can be appended by any letter from this alphabet, except the inverse of that letter. So, initially, there are $2g$ choices, after which there are $2g-1$ choices in each further step in constructing a reduced word in $\Gamma$. 

\smallskip
For any $\beta\in\Gamma$, observe that
\[
\ell(\beta^{-1}\gamma)\le \ell(\beta)+\ell(\gamma)
\]
which yields the desired upper bound by using the previous case and taking care of possible cancelling with suffixes of $\beta^{-1}$.
\end{proof}

Gerritzen and van der Put in \cite[I.4.1.3]{GvP1980} introduce the notion of \emph{good fundamental domain} for a $p$-adic Schottky group $\Gamma$, which is needed below. This is the complement in the projective line $\mathds{P}^1(K)$ of $2g$  open discs $B_1,C_1,\dots, B_g,C_g$ whose ``closures'' $B_i^+$, $C_i^+$ (i.e.\ where in the defining inequalities ``$<$'' is replaced with ``$\le$'', and radii are assumed to be in the valuation group of $K$) are mutually disjoint, and there are $g$ generators $\gamma_1,\dots,\gamma_g$ such that 
\[
\gamma_i(\mathds{P}^1(K)\setminus B_i)=C_i^+,\quad
\gamma_i(\mathds{P}^1(K)\setminus B_i^+)=C_i
\]
for $i=1,\dots,g$.
\newline

Let $\Omega(K)\subset\mathds{P}^1(K)$ be defined as
the complement of the set $\mathscr{L}\subset\mathds{P}^1(K)$ of limit points of the action of $\Gamma$, assuming that $\infty\in\mathscr{L}$.
Let $F=F(K)\subset\Omega(K)$ be a good fundamental domain for $\Gamma$.
Now, 
let $\alpha_g>0$
such that
\begin{align}\label{growthAssumption}
p^{f\alpha_g}>2g
\end{align}
and
define
\begin{align}\label{KernelFunction}
H_\alpha(\beta x,\gamma y)=\mu^\Gamma(F)^{-1}
\absolute{\pi}^{\alpha_g\ell(\beta^{-1}\gamma)}
\absolute{\beta x-\gamma y}^{-\alpha}
\end{align}
for $x,y\in F$, $\beta,\gamma\in\Gamma$,  and $\alpha>0$, and
where the  $\Gamma$-invariant Borel measure on $\Omega(K)\setminus V(\omega)$
evaluated on sets is as
\[
\mu^\Gamma(B)=\int_B\absolute{\omega}
\]
for any $\absolute{\omega}$-measurable set $B\subset\Omega(K)\setminus V(\omega)$.
\newline

Now, define the operator
\[
\mathcal{H}_\alpha u( \beta x)=\sum\limits_{\gamma\in\Gamma}
\int_{F}H_\alpha(\beta x,\gamma y)(u(y)-u(x))\absolute{\omega(y)}
\]
where $\omega\in\Omega^1(\Omega(K))^\Gamma$ is a $\Gamma$-invariant holomorphic differential $1$-form on $\Omega(K)$, and $u\in C\left(\Omega(K),\absolute{\cdot}_\infty\right)^\Gamma$ or $u\in L^2(\Omega(K),\absolute{\omega})$, and $x\in F$.
Observe that $\mathcal{H}_\alpha$ is an operator of the following:
\begin{align*}
\mathcal{H}_\alpha&\colon L^2(\Omega(K),\absolute{\omega})^\Gamma\to L^2(\Omega(K),\absolute{\omega})
\\
\mathcal{H}_\alpha&\colon
C(\Omega(K),\norm{\cdot}_\infty)^\Gamma\to C(\Omega(K),\norm{\cdot}_\infty)
\end{align*}
for $\alpha>0$.
Further, there is a bilinear
Dirichlet form
\begin{align*}
\mathscr{E}_\alpha(u,v)
&=\langle \mathcal{H}_\alpha u,\mathcal{H}_\alpha v\rangle_\omega
\\
&=\sum\limits_{\beta,\gamma\in\Gamma}
\int_F\int_F H_\alpha(\beta x,\gamma y)(u(y)-u(x))\left(\overline{v(y)}-\overline{v(x)}\right)\absolute{\omega(y)}\absolute{\omega(x)}
\end{align*}
and a quadratic Dirichlet form
\[
\mathscr{E}_\alpha(u)
=\langle\mathcal{H}^\alpha u,\mathcal{H}^\alpha u\rangle_\omega
\]
for $u,v\in L^2(\Omega(K),\absolute{\omega})^\Gamma$.

\begin{Lemma}\label{denselyDefinedOperator}
The operator $\mathcal{H}_\alpha$ is densely defined for $\alpha>0$.
\end{Lemma}

Denote the space of $\Gamma$-invariant locally constant functions on $\Omega(K)$ as
$\mathcal{D}(\Omega(K))^\Gamma$. 

\begin{proof}
Let $u\in\mathcal{D}(\Omega(K))^\Gamma$. Then
\begin{align*}
\mathcal{H}_\gamma^\alpha u(\beta x)&:=
\int_F H_\alpha(\beta x,\gamma y)(u(y)-u(x))\absolute{\omega(y)}
\\
&=\mu^\Gamma(F)^{-1}\absolute{\pi}^{\alpha_g\ell(\beta^{-1}\gamma)}
\int_F\absolute{\beta x-\gamma y}^{-\alpha}(u(y)-u(x))\absolute{\omega(y)}
\end{align*}
for $x\in F$.
Now, the distance between $x$ and $\gamma y$ can be arbitrarily large for fixed $x,y\in F$. It takes as values integer  powers of $\absolute{\pi}^{\alpha}$. Since $u$ is locally constant, it now follows  that the integral term in $\mathcal{H}^\alpha_\gamma u(\beta x)$ converges for all $\gamma\in\Gamma$ to a value bounded from above by a positive constant times a power of $\absolute{\pi}$.
By assumption (\ref{growthAssumption}), the number of $\gamma\in\Gamma\setminus\beta$ for which the values $\mathcal{H}_\gamma^\alpha u(x)$ are fixed, is bounded from above by $(2g)^{\ell(\beta)+\ell}$ for some natural $\ell>0$, cf.\ Lemma \ref{ReducedWordCount}.
Hence,
the infinite sum
\[
\mathcal{H}_\alpha u(\beta x)=\sum\limits_{\gamma\in\Gamma}\mathcal{H}_\gamma^\alpha u(\beta x)
\]
is bounded from above by a constant times a geometric series
in a power of 
\[
\frac{\absolute{\pi}^{\alpha_g}\absolute{\pi}^{\alpha}}{2g}<1
\]
and hence
converges for any $\beta\in\Gamma$, $x\in F$, and $\alpha>0$.
\end{proof}

\begin{Lemma}
The quadratic Dirichlet form  $u\mapsto \mathcal{E}_\alpha(u)$ is densely defined.
\end{Lemma}

\begin{proof}
Let $u\in\mathcal{D}(\Omega(K))^\Gamma$.
The value of $\mathcal{E}_\alpha(u)$
is
\begin{align*}
\mathcal{E}_\alpha(u)&=
\langle\mathcal{H}_\alpha u,\mathcal{H}_\alpha u\rangle_\omega
\\
&=\mu^\Gamma(F)^{-2}
\sum\limits_{\beta,\gamma\in\Gamma}
\absolute{\pi}^{2\alpha_g\ell(\beta^{-1}\gamma)}
\\
&\quad\cdot
\iint_{F^2}
\absolute{\beta x-\gamma y}^{-\alpha}\absolute{u(y)-u(x)}^2\absolute{\omega(y)}\absolute{\omega(x)}
\end{align*}
whose convergence is shown similarly as in the proof of Lemma \ref{denselyDefinedOperator}.
\end{proof}

Let $A=L^2(\Omega(K),\absolute{\omega})$
and
$\mathcal{H}_\alpha^*\colon A\to A^\Gamma$
the adjoint operator of $\mathcal{H}_\alpha\colon A^\Gamma\to A$.
Now, define
\[
\Delta_\alpha:=\mathcal{H}_\alpha^*\circ \mathcal{H}_\alpha\colon A^\Gamma\to A^\Gamma
\]
as an operator on $\Gamma$-invariant functions on $\Omega(K)$, or on functions on the Mumford curve $X(K)$,  which is the same thing.
There is also an operator
\[
\Delta_\alpha^\dagger:=\mathcal{H}_\alpha\circ\mathcal{H}_\alpha^*\colon A\to A
\]
for $\alpha>0$.

\begin{Lemma}
The operators $\mathcal{H}_\alpha$, $\mathcal{H}_\alpha^*$ are closed, the operators $\Delta_\alpha$ and $\Delta_\alpha^\dagger$ are self-adjoint, and the operators $I+\Delta_\alpha$, $I+\Delta_\alpha^\dagger$ have bounded inverses for $\alpha>0$.
\end{Lemma}

\begin{proof}
In order to see that $\mathcal{H}_\alpha$ is closed, assume $u_n\in\dom(\mathcal{H}_\alpha)$ such that $u_n\to u\in L^2(\Omega(K),\absolute{\omega})^\Gamma$, and
$\mathcal{H}_\alpha u_n\to v\in L^2(\Omega(K),\absolute{\omega})$.
Then
\begin{align*}
\norm{\mathcal{H}_\alpha u - v}_{\omega}
\le
\norm{\mathcal{H}_\alpha u-\mathcal{H}_\alpha u_n}_{\omega}
+\norm{\mathcal{H}_\alpha u_n-v}_{\omega}
\end{align*}
and the second summand tends to zero by assumption. The square of the first summand is 
\begin{align*}
\norm{\mathcal{H}_\alpha u-\mathcal{H}_\alpha u_n}_{\omega}^2
&=\mathcal{E}_\alpha(u-u_n)\to 0
\end{align*}
because
\[
\int_F\absolute{\beta x-\gamma y}^{-\alpha}(u_n(y)-u(y)+u_n(x)-u(x))\absolute{\omega(y)}
\]
tends to zero for $n\to\infty$  for all $\beta,\gamma\in\Gamma$, as $u-u_n$ tends to the constant zero function.
It follows that $\mathcal{H}_\alpha u=v\in L^2(\Omega(K),\absolute{\omega})$, i.e. $u\in\dom(\mathcal{H}_\alpha)$ for $\alpha>0$. The closedness of the adjoint is now a standard fact, and
the remaining assertions follow from von Neumann's Theorem on the adjoint \cite[p.\ 200]{Yoshida1980}. 
\end{proof}

A consequence is that it is also possible to write
\[
\mathcal{E}_\alpha(u)=\langle\mathcal{H}_\alpha u,\mathcal{H}_\alpha u\rangle_\omega=
\langle\Delta_\alpha u,u\rangle_\omega
\]
for $u\in\dom(\mathcal{E}_\alpha)$ using the self-adjoint operator $\Delta_\alpha$ on $L^2(\Omega(K),\absolute{\omega})^\Gamma$.

\section{Spectrum}

A Kozrev wavelet is a function
\[
\psi_{B,j}(x)
=\mu_K(B)^{-\frac12}
\chi(\pi^{d-1}\tau(j)\,x)\Omega(x\in B)
\]
where $B\subset K$ is a disc of radius $\absolute{\pi}^{-d}$, $d\in\mathds{Z}$, $j\in (O_K/\pi O_K)^\times$, and $\tau\colon O_K/\pi O_K\to K$ a lift.
They were introduced by S.\ Kozyrev as an eigenbasis in $L^2(\mathds{Q}_p,\absolute{dx})$ for the $p$-adic Vladimirov operator \cite{Kozyrev2002}.

\begin{Lemma}\label{oscillatoryIntegralCircle}
It holds true that
\[
\int_{\absolute{x}=\absolute{\pi}^k}
\chi(ax)\absolute{x}^m\absolute{dx}
=\begin{cases}
\absolute{\pi}^{k(m+1)}(1-\absolute{\pi}),&\absolute{a}\le\absolute{\pi}^{-k}
\\
-\absolute{\pi}^{k(m+1)+1},&\absolute{a}=\absolute{\pi}^{-k-1}
\\
0,&\text{otherwise}
\end{cases}
\]
for $k,m\in\mathds{Z}$.
\end{Lemma}

\begin{proof}
It holds true that
\[
\int_{\absolute{x}=\absolute{\pi}^k}
\chi(ax)\absolute{x}^m\absolute{dx}
=\absolute{\pi}^{km}
\int_{\absolute{x}=\absolute{\pi}^k}
\chi(ax)\absolute{dx}
\]
which shows how the assertion follows from the case $m=0$. That case is shown e.g.\ in \cite[Lem.\ 3.6]{Rogers2004} in the case $K=\mathds{Q}_p$. His proof carries over to general $K$ in a straightforward manner. 
\end{proof}

\begin{Lemma}\label{oscillatoryIntegralAroundZero}
Let $a\in K$ with $\absolute{a}=\absolute{\pi}^{d-1}$ for $d\in\mathds{Z}$, and let $m\in\mathds{N}$.
Then it holds true that
\[
\int_{\absolute{x}\le\absolute{\pi}^\ell}
\chi(ax)\absolute{x}^m
\absolute{dx}
=\begin{cases}
C(m)\absolute{\pi}^{\ell(m+1)} 
,&\ell\ge 1-d
\\
C(m)\absolute{\pi}^{(1-d)(m+1)}
-\absolute{\pi}^{1-d(m+1)}
,&\ell= -d
\\
0&\text{otherwise}
\end{cases}
\]
with
\[
C(m)=
\frac{\displaystyle
1-\absolute{\pi}}{\rule{0pt}{4mm}\displaystyle 1-\absolute{\pi}^{m+1}}
\]
In particular, the integral vanishes, if and only if $m=0$ and $\ell\le -d
$.
\end{Lemma}

\begin{proof}
It holds true that
\[
\int_{\absolute{x}=\absolute{\pi}^\ell}
\chi(ax)\absolute{x}^m\absolute{dx}
=\sum\limits_{k=\ell}^\infty\int_{\absolute{x}=\absolute{\pi}^k}
\chi(ax)\absolute{x}^m\absolute{dx}
\]
and according to Lemma \ref{oscillatoryIntegralCircle}, the right hand side vanishes, if and only if $d<-\ell$, as asserted. If $d>-\ell$, then
the right hand side equals
\[
\sum\limits_{k=\ell}^\infty
\absolute{\pi}^{k(m+1)}(1-\absolute{\pi})
=C(m)\absolute{\pi}^{\ell(m+1)}
\]
as asserted.
In the remaining case that $d=-\ell$, it holds true that the right hand side equals
\[
-\absolute{\pi}^{1-d(m+1)}+\sum\limits_{k=1-d}
\absolute{\pi}^{k(m+1)}
=C(m)\absolute{\pi}^{(1-d)(m+1)}-\absolute{\pi}^{1-d(m+1)}
\]
again asserted.
\end{proof}

\begin{Lemma}\label{localHaar}
Let $B=B_\ell(a)\subset K$ be a disc not containing the $r$ points $a_1,\dots,a_r\in K$. then the polynomial 
\[
h(x)=\prod\limits_{i=1}^r(x-a_i)
\]
restricted to $B$ has the constant absolute value
\[
\absolute{h|_B(x)}=\prod\limits_{i=1}^r\absolute{a-a_i}
\]
for $x\in B$.
\end{Lemma}

\begin{proof}
Since $a_1,\dots,a_r$ are not in $B$, it follows that
\[
\absolute{a-a_i}>\absolute{\pi}^\ell
\]
for all $i=1,\dots,r$. Hence,
for all $x\in B$, it holds true that
\[
\absolute{x-a_i}=\absolute{x-a+a-a_i}=\absolute{a-a_i}
\]
for $i=1,\dots,r$. This proves the assertion.
\end{proof}

\begin{Corollary}\label{vanishingWaveletMeanNiceCase}
Let $\psi_{B,j}$ be a  Kozyrev wavelet on $\Omega(K)$. Then
\[
\int_B\psi_{B,j}(x)\absolute{\omega(x)}
\]
vanishes if 
$B$ does not contain any zero of $\omega$.
\end{Corollary}

\begin{proof}
This follows immediately from Lemma \ref{localHaar} and the well-known result by Kozyrev, cf.\ 
\cite[Thm.\ 3.29]{XKZ2018}
or
\cite[Thm.\ 9.4.2]{AXS2010}.
\end{proof}

Let $\omega$ be a $\Gamma$-invariant regular $1$-form on $\Omega(K)$. Then, according to Lemma \ref{localHaar},
\begin{align}\label{constantFactor}
\absolute{\omega(x)}
=C_B\absolute{dx}
\end{align}
where
\[
C_B=C\cdot\prod\limits_{i=1}^r
\absolute{x-a_i}
\]
for some $C>0$ and $a_1,\dots,a_r\in F$ are the zeros of $\omega$ in $F$.

\begin{Corollary}\label{constantEqual}
It holds true that
\[
C_B=C_{\beta B}
\quad\text{and}\quad
\absolute{\beta'(x)}=1
\]
for all $\beta\in \Gamma$ and $x\in \Omega(K)\setminus V(\omega)$, where $B\subset \Omega(K)\setminus V(\omega)$ is a disc. 
\end{Corollary}

\begin{proof}
Assume w.l.o.g.\ that $B\subset F$.
The first statement now follows immediately from
\[
C_B\absolute{dx}=\absolute{\omega(x)}
=\absolute{\omega(\beta x)}
=C_{\beta B}\absolute{dx}
\]
for $\beta\in\Gamma$, because, since $\omega$ is $\Gamma$-invariant, $\beta B$ also does  not contain any zeros of $\omega$, and a similar reasoning as in the proof of Lemma \ref{localHaar} can be used. This also explains why the constant factor $C_{\beta B}$ exists in the first place.

\smallskip
Now, $\beta'$ does not have any zeros or poles in $B$, because the zeros and poles of $\beta'$ are zeros or poles of $\omega$. But $B$ is away from the zeros of $\omega$, and $\omega$ is a regular differential form, i.e.\  has no poles.
Hence, $\absolute{\beta'}$ is readily seen to be locally constant on $B$.
From the $\Gamma$-invariance of $\omega$, it follows that this is actually  constant equalling to one, because
\begin{align*}
C_{\tilde{B}}\,\mu_K
=\absolute{\omega}=
\absolute{\omega\circ\beta^{-1}}
=C_{\tilde{B}}\absolute{\beta'}^{-1}\mu_K
\end{align*}
as measures on $B$. This implies that $\absolute{\beta'(z)}=1$ for $z\in B$. But since $\Omega(K)\setminus V(\Omega)$ can be covered by discs, it follows that $\absolute{\beta'(x)}=1$ for all $x\in \Omega(K)\setminus V(\omega)$.
\end{proof}

\begin{remark}
If $B$ contains a zero of $\omega$, then it does happen that the $\absolute{\omega}$-mean of a Kozyrev wavelet supported in $B$ does not vanish, as can be seen in the case of $0\in V(\omega)$ and $B$ a small disc containing $0$, by using Lemma \ref{oscillatoryIntegralAroundZero}. However, it is not clear to the author whether this holds true in all cases, i.e.\ the converse implication in Corollary \ref{vanishingWaveletMeanNiceCase} might possibly not hold true.
\end{remark}

\begin{Lemma}\label{MoebiusTransformDistance}
Let $\gamma\in\Gamma$. Then
\[
\gamma(x)-\gamma(y)=
\gamma'(x)^{\frac12}\gamma'(y)^{\frac12}(x-y)
\]
for a suitable choice of square root in $K$.
\end{Lemma}

\begin{proof}
Assume that $\gamma\in\Gamma$ is represented by a matrix
\[
A=\begin{pmatrix}a&b\\c&d\end{pmatrix}\in \SL_2(K)
\]
Then
\[
\gamma'(z)=\frac{1}{(cz+d)^2}
\]
and
\begin{align*}
\gamma(x)-\gamma(y)&=
\frac{ax+b}{cx+d}-\frac{ay+b}{cy+d}
=\frac{(ax+b)(cy+d)-(ay+b)(cx+d)}{(cx+d)(cy+d)}
\\
&=\frac{(ad-bc)(x-y)}{(cx+d)(cy+d)}
=\gamma'(x)^{\frac12}\gamma'(y)^\frac12(x-y)
\end{align*}
for a suitable choice of square roots in $K$, as asserted.
\end{proof}



\begin{Lemma}\label{invariantKernel}
Let $x,y\in\Omega(K)\setminus V(\omega)$.
Then
\begin{align*}
\absolute{\beta x-\gamma y}&=\absolute{x-\beta^{-1}\gamma y}
\end{align*}
for $\beta,\gamma\in\Gamma$.
\end{Lemma}

\begin{proof}
Let $x,y\in \Omega(K)\setminus V(\omega)$. Then
\begin{align*}
\absolute{x-\beta^{-1}\gamma y}
&=\absolute{\beta^{-1}\beta x-\beta^{-1}\gamma y}
\\
&=\absolute{\beta x-\gamma y}\absolute{\beta'(\beta x)}^{-\frac12}\absolute{\beta'(\gamma y)}^{-\frac12}
\end{align*}
for $\beta,\gamma\in\Gamma$. Hence, the assertion follows from Corollary \ref{constantEqual}.
\end{proof}



\begin{Lemma}\label{eigenfunction}
It  holds true that
\begin{align*}
\int_F\absolute{x-y}^{-\alpha}&(\psi_{B,j}(y)-\psi_{B,j}(x))\absolute{\omega(y)}
\\
&=
-\left(
\int_{F\setminus B}\absolute{x-y}^{-\alpha}\absolute{\omega(y)}
+\mu^\Gamma(B)^{1-\alpha}
\right)
\psi_{B,j}(x)
\end{align*}
for $x\in K$, $B\subset \Omega(K)\setminus V(\omega)$ a disc, and $j\in O_K/\pi O_K$.
\end{Lemma}

\begin{proof}
This follows from \cite[Thm.\ 3]{Kozyrev2004}, as the conditions for that theorem to be valid are satisfied.
\end{proof}

A Kozyrev wavelet $
\psi_{B,j}(x)$ supported on a disc $B\subset F$ can be extended to a $\Gamma$-invariant function
\[
\psi_{B,j}^\Gamma(\gamma x):=\psi_{B,j}(x)
\]
for all $\gamma\in\Gamma$.
Call this function a \emph{$\Gamma$-invariant Kozyrev wavelet}.
\newline

Define the number
\[
N_F(B):=
\absolute{\mathset{\text{discs $\tilde B\subset F$}
\mid\mu^\Gamma(\tilde{B})=\mu^\Gamma(B)\;\wedge
\;\forall\gamma\in\Gamma\colon I_F(\gamma\tilde{B})=I_F(\gamma B)
}}
\]
for a given disc $B\subset F$ and
\[
I_F(\gamma B):=\int_{F\setminus B}\absolute{x-\gamma y}^{-\alpha}\absolute{\omega(y)}
\]
for $\gamma\in\Gamma$, $\alpha>0$.

\begin{thm}\label{spectrum}
The space $L^2(\Omega(K)\setminus V(\omega),\absolute{\omega})^\Gamma$ of $\Gamma$-invariant $L^2$-functions
on $\Omega(K)\setminus V(\omega)$
has an orthonormal basis consisting of the $\Gamma$-periodic wavelets $\psi_{B,j}^\Gamma$ supported in $\Omega(K)\setminus V(\omega)$, and these are eigenfunctions of $\Delta_\alpha^\frac12$ for $\alpha>0$. 
The eigenvalue 
corresponding to $\psi_{B,j}^\Gamma$ is
\[
\lambda_B=
\mu^\Gamma(F)^{-1}
\sum\limits_{\gamma\in\Gamma}
\absolute{\pi}^{\alpha_g\ell(\gamma)}
\left(
\int_{F\setminus B}\absolute{x-\gamma y}^{-\alpha}\absolute{\omega(y)}
+
\mu^\Gamma(B)^{1-\alpha}
\right)
\]
for $j\in O_K/\pi O_K\setminus\mathset{0}$,  $B\subset F\setminus V(\omega)$ a disc whose $\Gamma$-translates form the support of $\psi_{B,j}$, and 
 $F$ a good fundamental domain for the action of $\Gamma$. The multiplicity of eigenvalue $\lambda_B$ is
$N_F(B)\cdot(p^f-1)
$. Both, $\lambda_B$ and its multiplicity, are invariant under replacing $F$ with $\gamma F$ for any $\gamma\in \Gamma$.
The restriction of $\Delta_\alpha^\frac12$ to $L^2(\Omega(K)\setminus V(\omega),\absolute{\omega})^\Gamma$ coincides with $-\mathcal{H}_\alpha$ for $\alpha>0$.
\end{thm}

This is Theorem 1.

\begin{proof}
The $\Gamma$-invariant function $\psi_{B,j}^\Gamma(x)$ is an element of $L^2(\Omega(K)\setminus V(\omega),\absolute{\omega})^\Gamma$, because 
\begin{align*}
\int_{\Omega(K)}\psi_{B,j}^\Gamma(x)\,\absolute{\omega(x)}
&=\sum\limits_{\gamma\in\Gamma}\int_{B}\psi_{B,j}^\Gamma(\gamma x)\absolute{\omega(x)}
=\sum\limits_{\gamma\in\Gamma}
\int_B\psi_{B,j}(x)\absolute{\omega(x)}=0
\end{align*}
where the last equation holds true by Corollary \ref{vanishingWaveletMeanNiceCase}.
Since any $\Gamma$-periodic $L^2$-function on $\Omega(K)\setminus V(\omega)$ has to  have mean zero, it now follows that the space of $\Gamma$-periodic $L^2$-functions on that space is spanned by the $\Gamma$-invariant Kozyrev wavelets supported in $\Omega(K)\setminus V(\omega)$, as these are in $1-1$-correspondence with the Kozyrev wavelets which are an orthonormal basis of $L^2(F\setminus V(\omega),\mu_K)$. Notice that the measure $\absolute{\omega}$ differs from $\mu_K$ on the support of any Kozyrev wavelet only by a constant factor according to (\ref{constantFactor}).
Therefore, the different choices of measures for those Hilbert spaces are not an issue.

\smallskip
Now, let $\beta,\gamma\in\Gamma$. Then
\begin{align*}
\mu^\Gamma(F)&\absolute{\pi}^{-\alpha_g\ell(\beta^{-1}\gamma)}\,\mathcal{H}_\gamma^\alpha\psi_{B,j}^\Gamma(\beta x)
\\
&\stackrel{\text{Lem.\ \ref{invariantKernel}}}{=}
\int_{F}\absolute{x-\beta^{-1}\gamma y}^{-\alpha}
\left(\psi_{B,j}^\Gamma(y)-\psi_{B,j}^\Gamma(x)\right)\absolute{\omega(y)}
\\
&=-\left(\int_{F\setminus B}\absolute{x-\beta^{-1}\gamma y}^{-\alpha}\absolute{\omega(y)}+\mu^\Gamma(B)^{1-\alpha}\right)\psi_{B,j}^{\Gamma}(x)
\end{align*}
for $x\in F$, where 
the last equality uses \cite[Thm.\ 3.1]{Kozyrev2004} in a similar manner as Lemma \ref{eigenfunction}.
%

\smallskip
What has been established so far, is that $\psi_{B,j}^\Gamma\in L^2(\Omega(K)\setminus V(\omega),\absolute{\omega})^\Gamma$ is an eigenfunction of $\mathcal{H}_\gamma^\alpha$ for any $\gamma\in\Gamma$ and $\alpha>0$. 
This means that $\mathcal{H}_\alpha$ takes
the closed subspace $L^2(\Omega(K)\setminus V(\omega),\absolute{\omega})^\Gamma$ to itself. Hence, $\Delta_\alpha$ equals the square of the restriction of $\mathcal{H}_\alpha$ to that space, since the $\Gamma$-invariant Kozyrev eigenvalues of $\mathcal{H}_\alpha$ are real numbers. 
\newline

Now, it follows that
\begin{align*}
\mathcal{H}_\alpha\,&\psi_{B,j}^\Gamma(\beta x)
\\
&=-\mu^\Gamma(F)^{-1}
\sum\limits_{\gamma\in\Gamma}
\absolute{\pi}^{\alpha_g\ell(\beta^{-1}\gamma)}
\left(
\int_{F\setminus B}\absolute{x-\beta^{-1}\gamma y}^{-\alpha}\absolute{\omega(y)}
+\mu^\Gamma(B)^{1-\alpha}
\right)\psi_{B,j}^\Gamma(x)
\\
&=-\mu^\Gamma(F)^{-1}
\sum\limits_{\gamma\in\Gamma}
\absolute{\pi}^{\alpha_g\ell(\gamma)}
\left(
\int_{F\setminus B}\absolute{x-\gamma y}^{-\alpha}\absolute{\omega(y)}
+\mu^\Gamma(B)^{1-\alpha}
\right)\psi_{B,j}^\Gamma(x)
\end{align*}
where the last equality follows from the fact that summation over $\gamma\in\Gamma$ is the same as summation over $\beta^{-1}\gamma\in\Gamma$. Hence, the expression does not depend on the choice of $\beta\in\Gamma$. 
Hence, the $\psi_{B,j}^\Gamma$ is an  eigenfunction of $\Delta_\alpha^\frac12$ for $\alpha>0$ with  eigenvalue $\lambda_B$ as stated. 
Indeed, it can be checked that the infinite sum does converge, because $\gamma y$ never falls into $B$, implying that $\absolute{x-\Gamma y}$ does not become arbitrarily small. This proves the value of eigenvalue $-\lambda_B$
of $\mathcal{H}_\alpha$, or, equivalently, that of eigenvalue $\lambda_B$ of $\Delta_\alpha^\frac12$.
Hence, 
\[
\Delta_\alpha^{\frac12}=-\mathcal{H}_\alpha
\]
for $\alpha>0$, as asserted.
\newline

As to the multiplicities,  clearly, $\lambda_B$ does not depend on the choice of $j\in O_K/\pi O_K$. This accounts for the factor $(p^f-1)$ in its multiplicity. The other factor is obtained by observing that $\lambda_B$ only depends on the $\Gamma$-invariant volume of disc $B\subset F$  and a summation of  $I_F(\gamma B)$ terms, which is invariant by Lemma \ref{invariantKernel}. This again yields a finite contribution to the multiplicity, as $F$ is compact, and it also follows that both, $\lambda_B$ and its multiplicity, are invariant under replacing $F$ with $\gamma F$ for $\gamma\in\Gamma$. 
%
This proves the theorem. 
\end{proof}

The eigenvalue of $\psi_{B,j}$ 
is invariant under the action of $\Gamma$, but there is a dependence on the 
choice of a good fundamental domain modulo $\Gamma$-equivalence,
which likely leads to different spectra for different such choices.
Anyway, if $\phi\colon F\to\tilde{F}$ is a bianalytic map between two fundamental domains, then
\[
H_\alpha(\beta\phi(x),\gamma\phi(y))
=\mu^\Gamma(\tilde{F})^{-1}
\absolute{\pi}^{\alpha_g\ell(\beta^{-1}\gamma)}
\absolute{\beta\phi(x)-\gamma\phi(y)}^{-\alpha}
\]
and
\[
\absolute{\omega(\phi(y))}=\absolute{f(\phi(y))}\absolute{\phi'(y)}\absolute{dy}
\]
where $\omega$ on $\tilde{F}$ takes the form:
\[
\omega(z)=f(z)\,dz
\]
for some holomorphic function $f\colon F\to K$. This leads to
\begin{align}\label{trafo}
\mathcal{H}_\alpha 
&
u(\beta\phi(x))
\\\nonumber
&=\mu^\Gamma(\tilde{F})^{-1}
\sum\limits_{\gamma\in\Gamma}
\absolute{\pi}^{\alpha_g\ell(\beta^{-1}\gamma)}
\int_{F}\absolute{\beta\phi(x)-\gamma\phi(y)}^{-\alpha}
\absolute{f(\phi(y))}\absolute{\phi'(y)}\absolute{dy}
\end{align}
for functions $u\colon\tilde{F}\to\mathds{C}$, and $\beta\in\Gamma$.

\begin{Lemma}
The quantity $\lambda_B$ corresponding to $\psi_{B,j}(x)$ transforms under $\phi$ to $\lambda_{\phi(B)}$ with
\begin{align*}
\lambda_{\phi(B)}&=
\mu^\Gamma(\phi(F))^{-1}
\\
&\cdot\sum\limits_{\gamma\in \Gamma}
\absolute{\pi}^{\alpha_g\ell(\gamma)}
\left(
\int_{F\setminus B}\absolute{x-\gamma\phi(z)}^{-\alpha}\absolute{\phi'(z)}\absolute{\omega(z)}-\mu^\Gamma(\phi(B))
\right)
\end{align*}
where $B\subset\Omega(K)\setminus V(\omega)$ is a disc.
\end{Lemma}

\begin{proof}
Since the bi-analytic pre-image of a $p$-adic disc is a $p$-adic disc, the expression $\lambda_{\phi(B)}$ is a well-defined eigenvalue of a well-defined $\Gamma$-periodic wavelet.
The expression for  $\lambda_{\phi(B)}$ follows in a straightforward manner. 
\end{proof}

\begin{remark}
Both, the genus and the geometry of a Mumford curve are encoded in the spectrum of $-\Delta_\alpha^\frac12$, as can be seen in Theorem \ref{spectrum}. Firstly, via the number of elements of $\Gamma$ of a given length, leading to a given coefficient in the infinite sum making up $\lambda_B$. This coefficient thus depends on the number $g$ of free generators of $\Gamma$, i.e.\ the genus of $X$. Secondly, via the integral $\int_{F\setminus B}$
which is  determined by  the geometry of a Mumford curve via the holes in a good fundamental domain.
\end{remark}
\section{Feller property}





\begin{Lemma}\label{FellerSemigroup}
The linear  operator $\mathcal{H}_\alpha=-\Delta_\alpha^\frac12$ generates a Feller semigroup $\exp\left(-t\Delta_\alpha^\frac12\right)$ with $t\ge 0$ on $C(\Omega(K),\norm{\cdot}_\infty)^\Gamma$ for $\alpha>0$.
\end{Lemma}

\begin{proof}
The criteria given by the Hille-Yosida-Ray Theorem 
are verified, cf.\ \cite[Ch.\ 4, Lem.\ 2.1]{EK1986}.
\newline

\noindent
1. The domain of $-\Delta_\alpha^\frac12$ is dense in $C(\Omega(K),\mathds{R})^\Gamma$. This follows from 
Lemma \ref{denselyDefinedOperator}.
\newline

\noindent
2. $-\Delta_\alpha^\frac12$ satisfies the positive maximum principle. For this, let $h\in\mathcal{D}(\Omega(K))^\Gamma$, and $x_0\in\Omega(K)$ such that $h(x_0)=\sup\limits_{x\in\Omega(K)}h(x)$. This exists, because $h$ is $\Gamma$-periodic, and the fundamental domain $F$ is compact.
Then
\[
-\Delta_\alpha^\frac12h(x_0)\le\int_{\Omega(K)}H_\alpha(x_0,y)(h(x_0)-h(x_0))\absolute{\omega(y)}\le0
\]
which implies the positive maximum principle.
\newline

\noindent
3. $\range(\eta I+\Delta_\alpha)$ is dense in $C(\Omega(K),\mathds{R})^\Gamma$ for some $\eta>0$.
Since $-\Delta_\alpha^{\frac12}$ is unbounded, an approach different the proof of \cite[Lem.\ 4.1]{ZunigaNetworks}
is required. Let $h\in C(\Omega(K),\mathds{R})$, $\eta>0$.
The task is to find a solution of the equation
\begin{align}\label{resolventEquation}
\left(\eta I+\Delta_\alpha^\frac12\right)u=h
\end{align}
for some $\eta>0$ and $h$ in some dense subspace of $C(\Omega(K),\mathds{R})^\Gamma$.
The equation formally can be rewritten as
\begin{align}\label{rewrittenOperator}
u(z)-\frac{\int H_\alpha(z,y)u(y)\absolute{\omega(y)}}{\eta+\deg(z)}=\frac{h(z)}{\eta+\deg(z)}
\end{align}
with
\[
\deg(z)=\int_{\Omega(K)}H_\alpha(z,y)\absolute{\omega(y)}
\]
which does not converge, as the operator $\Delta_\alpha^\frac12$ is unbounded. That is why  the operator
\[
T_ku(z)
=\frac{\int_{\Omega_{z,k}}H_\alpha(z,y)u(y)\absolute{\omega(y)}}{\eta+\deg_k(z)}
\]
with
\[
\Omega_{z,k}=\bigsqcup\limits_{\gamma\in\Gamma}\gamma F_{z,k}
\]
and
\[
F_{z,k}=F\setminus B_k(z)
\]
for $k>>0$
is now being studied.
Let
\[
\deg_k(z)=\int_{\Omega_{z,k}}H_\alpha(z,y)\absolute{\omega(y)}
\]
which is finite for $k>>0$, and
\[
\absolute{T_ku(z)}\le
\frac{\deg_k(z)}{\eta+\deg_k(z)}\norm{u}_\infty
\]
where the supremum norm of $u$ is finite, as $u$ is $\Gamma$-invariant and $F$ is compact.
Hence,
\[
\norm{T_k}\le\frac{1}{\eta/\deg_k(z)+1}<1
\]
for any $\eta>0$, and $k>>0$, and in this case it follows that
$I+T_k$ has a bounded inverse as an operator on $C(\Omega(K),\mathds{R})^\Gamma$. This proves the denseness of its range for $k>>0$.
\newline

Now, let $h\in\mathcal{D}(\Omega(K))^\Gamma$, and $u_k,u_\ell\in C(\Omega(K),\mathds{R})^\Gamma$ be solutions of
\[
(I+T_k)u_k=\frac{h}{\eta+\deg_k}
,\quad(I+T_\ell)u_\ell=\frac{h}{\eta+\deg_\ell}
\]
for $k,\ell>>0$. Then
\begin{align}\label{complicated}
u_k-u_\ell=
\frac{(I+T_\ell)(\eta+\deg_\ell)-(I+T_k)(\eta+\deg_k)}{(I+T_k)(I+T_\ell)(\eta+\deg_k)(\eta+\deg_\ell)}h
\end{align}
shows that $u_k$ is a Cauchy sequence w.r.t.\ $\norm{\cdot}_\infty$. The reason is that, firstly,
\[
\norm{T_k}=\sup\limits_{z\in F}\frac{\deg_k(z)}{\eta+\deg_k(z)}
\]
clearly holds true, and this is a (strictly increasing) sequence convergent to $1$, and 
this implies the convergence of the sequence of operators $T_k$ to a bounded linear operator $T$ on $C(\Omega(K),\mathds{R})^\Gamma$.
Secondly, the numerator of the right hand side of (\ref{complicated}) is
\[
\eta(T_\ell-T_k)+(\deg_\ell-\deg_k)+(T_\ell\deg\ell-T_k\deg_k)
\]
whose first and second terms become arbitrarily small in norm as $\ell\ge k\to\infty$. The third term is
\[
T_\ell\deg_\ell-T_k\deg_k
=(T_\ell\deg_\ell-T_k\deg_\ell)+(T_k\deg_\ell-T_k\deg_k)
\]
both of whose terms in norm become arbitrarily small as $\ell\ge k\to\infty$.
Hence, $u_k$ converges to some $u\in C(\Omega(K),\mathds{R})^\Gamma$ which is seen to be a solution of
(\ref{resolventEquation}) by using the limit operator $T$ as follows: Namely, 
\[
(\eta+\deg_k)T_k\to(\eta+\deg)T\quad(k\to\infty)
\]
where the limit operator
coincides with the unbounded integral operator
\[
u\mapsto A u=\int_{\Omega(K)}
H_\alpha(\cdot,y)u(y)\absolute{\omega(y)}
\]
which shows that the
operator
\[
\frac{A}{\eta+\deg}=T
\]
appearing in (\ref{rewrittenOperator}) is bounded.
Now, $u_k$ is a solution of
\[
(\eta I -H_k)u_k=h
\]
with 
\[
H_k=(\eta+\deg_k)T_k-\deg_k
\]
which for $k\to\infty$ converges to $-\Delta_\alpha^\frac12$. As $u_k\to u$, it follows that
\[
(\eta I+\Delta_\alpha^\frac12)u=(\eta I-H_k)u+(H_k+\Delta_\alpha^\frac12)u
\]
where
\[
(\eta I-H_k)u
=(\eta I-H_k)u_k+H_k(u_k-u)
=h+H_k(u_k-u)\to h
\]
and
\[
(H_k+\Delta_\alpha^\frac12)u\to 0
\]
for $k\to\infty$.
Hence, $u$ is a solution of (\ref{resolventEquation}).
This proves that the range of $\eta I+\Delta_\alpha^\frac12$ contains the real-valued functions in $\mathcal{D}(\Omega(K))^\Gamma$ which is dense in $C(\Omega,\mathds{R})^\Gamma$. 
\newline

Now, by Hille-Yosida-Ray, the assertion follows.
\end{proof}

\begin{thm}
There exists a probability measure $p_t(x,\cdot)$ with $t\ge0$, $x\in \Omega(K)\setminus V(\omega)$ on the Borel $\sigma$-algebra of $\Omega(K)\setminus V(\omega)$ such that the Cauchy problem for the heat equation 
\[
\left(
\frac{\partial}{\partial t}+\Delta_\alpha^\frac12
\right)h(t,x)=0
\]
having initial condition $h(0,x)=h_0(x)\in C(\Omega(K)\setminus V(\omega),\norm{\cdot}_\infty)^\Gamma$ has a unique solution in $C^1\left((0,\infty),\Omega(K)\setminus V(\omega)\right)^\Gamma$ of the form
\[
h(t,x)=\int_{\Omega(K)\setminus V(\omega)}h_0(y)p_t(x,\absolute{\omega(y)}
\]
Additionally, $p_t(x,\cdot)$ is the transition function of a strong Markov process on $(\Omega(K)\setminus V(\omega))/\Gamma$ whose paths are c\`adl\`ag.
\end{thm}

This is Theorem 2.
The notation $C^1\left((0,\infty),\Omega(K)\setminus V(\omega)\right)^\Gamma$ indicates that for each $t>0$, any such function $h(t,x)$ is $\Gamma$-invariant.

\begin{proof}
According to Lemma \ref{FellerSemigroup}, $-\Delta_\alpha$ generates a Feller semigroup on the Banach space $C(\Omega(K),\norm{\cdot}_\infty)^\Gamma$. Using 
\cite[Prop.\ 15]{brad_HeatMumf} allows to restrict to
the the closed subspace
$C(\Omega(K)\setminus V(\omega),\norm{\cdot}_\infty)^\Gamma$ invariant under $-\Delta_\alpha^\frac12$.
Hence, $-\Delta_\alpha^\frac12$ generates a Feller semigroup also on that space. Now, it can be argued as in the proof of \cite[Thm.\ 4.2]{ZunigaNetworks}, namely that there exists a uniformly stochastically continuous $C_0$-transition
function $p_t(x, \absolute{\omega(y)})$ satisfying condition (L) of \cite[Thm.\ 2.10]{Taira2009}
 such that
 \[
 e^{-t\Delta_\alpha^\frac12}h_0(x)=
 \int_{\Omega(K)\setminus V(\omega)}
 h_0(y)p_t(x,\absolute{\omega(y)})
 \]
 cf.\ \cite[Thm.\ 2.15]{Taira2009}.
 From the correspondence between transition functions and Markov processes, there now exists a strong Markov process on the quotient space $(\Omega(K)\setminus V(\omega))/\Gamma$, which consists of $X(K)$ minus finitely many points, and whose paths are c\`adl\`ag, cf.\ \cite[Thm.\ 2.12]{Taira2009}.
\end{proof}

\section*{Acknowledgements}

\'Angel Mor\'an Ledezma, Leon Nitsche and David Weisbart are warmly thanked for fruitful discussions. Wilson Z\`u\~n\`{\i}ga-Galindo is thanked for bringing this subject onto the table. The anonymous referee is thanked for pointing out errors and for valuable suggestions for improving the exposition of the paper.
This work is partially supported by the Deutsche Forschungsgemeinschaft under project number 469999674.

\bibliographystyle{plain}
\bibliography{biblio}

\begin{thebibliography}{10}

\bibitem{AXS2010}
S.~Albeverio, A.~Yu. Khrennikov, and V.~M. Shelkovich.
\newblock {\em Theory of $p$-adic distributions: linear and non-linear models}.
\newblock London Mathematical Society Lecture Note Series, 370. Cambridge
  University Press, Cambridge, 2010.

\bibitem{brad_HeatMumf}
P.E. Bradley.
\newblock Heat equations and wavelets on {Mumford} curves and their finite
  quotients.
\newblock {\em Journal of Fourier Analysis and Applications}, 29:62, 2023.

\bibitem{brad_thetaDiffusionTateCurve}
P.E. Bradley.
\newblock Theta-induced diffusion on {Tate} elliptic curves over
  non-archimedean local fields.
\newblock arXiv:2312.03570 [math.NT], 2023.

\bibitem{Brad_HearingGenusMumf}
P.E. Bradley.
\newblock Heat equations and hearing the genus on $p$-adic {Mumford} curves via
  automorphic forms.
\newblock arXiv:2402.02869, 2024.

\bibitem{EK1986}
S.N. Ethier and T.G. Kurtz.
\newblock {\em Markov Processes - Characterization and Convergence}.
\newblock Wiley Series in Probability and Mathematical Statistics. John Wiley
  \& Sons, New York, 1986.

\bibitem{GvP1980}
L.~Gerritzen and M.~{van der Put}.
\newblock {\em Schottky Groups and Mumford Curves}.
\newblock Lecture Notes in Mathematics, vol. 817. Springer, Heidelberg, New
  York, 1980.

\bibitem{IgusaLocalZeta}
J.~Igusa.
\newblock {\em An introduction to the theory of local {Zeta} functions},
  volume~14 of {\em AMS/IP studies in advanced mathematics}.
\newblock American Mathematical Society, International Press, 2002.

\bibitem{XKZ2018}
A.~Khrennikov, S.~Kozyrev, and W.~A. Zúñiga-Galindo.
\newblock {\em Ultrametric Equations and its Applications}.
\newblock Encyclopedia of Mathematics and its Applications (168). Cambridge
  University Press, 2018.

\bibitem{Kochubei2018}
A.N. Kochubei.
\newblock Linear and nonlinear heat equations on a $p$-adic ball.
\newblock {\em Ukr. Math. J.}, 70:217--231, 2018.

\bibitem{Kozyrev2002}
S.~V. Kozyrev.
\newblock Wavelet theory as $p$-adic spectral analysis.
\newblock {\em Izvestiya: Mathematics}, 66(2):367--376, 2002.

\bibitem{Kozyrev2004}
S.V. Kozyrev.
\newblock $p$-adic pseudodifferential operators and $p$-adic wavelets.
\newblock {\em Theoretical and Mathematical Physics}, 138(3):322--332, 2004.

\bibitem{Rogers2004}
K.M. Rogers.
\newblock {\em Real and $p$-adic oscillatory integrals}.
\newblock PhD thesis, The University of New South Wales Sydney, 2004.

\bibitem{Schneider2011}
P.~Schneider.
\newblock {\em $p$-Adic Lie Groups}.
\newblock Grundlehren der mathematischen Wissenschaften 344. Springer-Verlag,
  Berlin Heidelberg, 2011.

\bibitem{Serre1992}
J.-P. Serre.
\newblock {\em Lie Algebras and Lie Groups}.
\newblock LNM 1500. Springer-Verlag, Berlin Heidelberg, 1992.

\bibitem{Taibleson1975}
M.H. Taibleson.
\newblock {\em Fourier analysis on local fields}.
\newblock Princeton University Press, Princeton, N.J., University of Tokyo
  Press, Tokyo, 1975.

\bibitem{Taira2009}
K.~Taira.
\newblock {\em Boundary Value Problems and Markov Processes}.
\newblock Lecture Notes in Mathematics, vol. 1499. Springer-Verlag, second
  edition, 2009.

\bibitem{Urban2024}
R.~Urban.
\newblock The {Vladimirov} operator with variable coefficients on finite adeles
  and the {Feynman} formulas for the {Schrödinger} equation.
\newblock {\em J. Math. Phys.}, 65(4):042103, 2024.

\bibitem{VVZ1994}
V.S. Vladimirov, I.V. Volovich, and E.I. Zelenov.
\newblock {\em $p$-adic Analysis and Mathematical Physics}.
\newblock Series on Soviet and East European Mathematics, 1. World Scientific
  Publishing Co., Inc., River Edge, NJ, 1994.

\bibitem{WeilAAG}
A.~Weil.
\newblock {\em Ad\`eles and Algebraic Groups}.
\newblock Progress in Mathematics 23. Birkh\"auser, Boston, 1982.

\bibitem{Weisbart2021}
D.~Weisbart.
\newblock On infinitesimal generators and {Feynman–Kac} integrals of adelic
  diffusion.
\newblock {\em Journal of Mathematical Physics}, 62(6), 2021.

\bibitem{Weisbart2024}
D.~Weisbart.
\newblock $p$-adic {Brownian} motion is a scaling limit.
\newblock {\em J. Phys. A: Math. Theor.}, 57:205203, 2024.

\bibitem{Yasuda2017}
T.~Yasuda.
\newblock The wild {McKay} correspondence and $p$-adic measures.
\newblock {\em J. Eur. Math. Soc.}, 19:3709--3734, 2017.

\bibitem{Yoshida1980}
K.~Yoshida.
\newblock {\em Functional Analysis}.
\newblock Springer, sixth edition, 1980.

\bibitem{ZunigaNetworks}
W.~{Z\'{u}\~{n}iga-Galindo}.
\newblock Reaction-diffusion equations on complex networks and {Turing}
  patterns, via $p$-adic analysis.
\newblock {\em Journal of Mathematical Analysis and Applications},
  491(1):124239, 2020.

\bibitem{Zuniga2022}
W.A. Zúñiga-Galindo.
\newblock Ultrametric diffusion, rugged energy landscapes and transition
  networks.
\newblock {\em Physica A: Statistical Mechanics and its Applications},
  597:127221, 2022.

\end{thebibliography}

\end{document}